\documentclass[a4paper,11pt]{article}
\usepackage[latin1]{inputenc}
\usepackage[english]{babel}
\usepackage{amsmath}
\usepackage{amsfonts}
\usepackage{amssymb}
\usepackage{epsfig}
\usepackage{amsopn}
\usepackage{amsthm}
\usepackage{color}
\usepackage{graphicx}
\usepackage{enumerate}
\newtheorem{theorem}{Theorem}[section]
\newtheorem{corollary}[theorem]{Corollary}
\newtheorem{lemma}[theorem]{Lemma}
\newtheorem{proposition}[theorem]{Proposition}
\newtheorem{definition}[theorem]{Definition}

\newtheorem*{theorem*}{Theorem}
\newtheorem*{lemma*}{Lemma}
\newtheorem*{remark*}{Remark}
\newtheorem*{definition*}{Definition}
\newtheorem*{proposition*}{Proposition}
\newtheorem*{corollary*}{Corollary}
\numberwithin{equation}{section}
\newcommand{\real}{\mathbb{R}}

\def\a{\alpha}

\def\e{\varepsilon}        

\def\qed{\,\unskip\kern 6pt \penalty 500
\raise -2pt\hbox{\vrule \vbox to8pt{\hrule width 6pt
\vfill\hrule}\vrule}\par}
\definecolor{darkblue}{rgb}{0.05, .05, .65}
\definecolor{darkgreen}{rgb}{0.1, .65, .1}
\definecolor{darkred}{rgb}{0.8,0,0}
\newcommand{\beqn}{\begin{equation}}
\newcommand{\eeqn}{\end{equation}}
\newcommand{\bear}{\begin{eqnarray}}
\newcommand{\eear}{\end{eqnarray}}
\newcommand{\bean}{\begin{eqnarray*}}
\newcommand{\eean}{\end{eqnarray*}}
\begin{document}

\title{\huge \bf Asymptotic behavior for the heat equation in nonhomogeneous media with critical density}

\author{
\Large Razvan Gabriel Iagar\,\footnote{Dept. de An\'{a}lisis Matem\'{a}tico,
Universitat de Valencia, Dr. Moliner 50, 46100, Burjassot
(Valencia), Spain, \textit{e-mail:} razvan.iagar@uv.es},
\footnote{Institute of Mathematics of the Romanian Academy, P.O. Box
1-764, RO-014700, Bucharest, Romania.}
\\[4pt] \Large Ariel S\'{a}nchez,\footnote{Departamento de Matem\'{a}tica Aplicada,
Universidad Rey Juan Carlos, M\'{o}stoles, 28933, Madrid, Spain,
\textit{e-mail:} ariel.sanchez@urjc.es}\\ [4pt] }
\date{\today}
\maketitle

\begin{abstract}
We study the large-time behavior of solutions to the heat equation
in nonhomogeneous media with critical singular density
$$
|x|^{-2}\partial_{t}u=\Delta u, \quad \hbox{in} \
\real^N\times(0,\infty)
$$
in dimensions $N\geq3$. The asymptotic behavior proves to have some
interesting and quite striking properties. We show that there are
two completely different asymptotic profiles depending on whether
the initial data $u_0$ vanishes at $x=0$ or not. Moreover, in the
former the results are true only for radially symmetric solutions,
and we provide counterexamples to convergence to symmetric profiles
in the general case.
\end{abstract}

\vspace{2.0 cm}

\noindent {\bf AMS Subject Classification 2010:} 35B33, 35B40,
35K05, 35Q79.

\medskip

\noindent {\bf Keywords and phrases:} heat equation, non-homogeneous
media, singular density, asymptotic behavior, radially symmetric
solutions, thermal propagation.

\section{Introduction}

The aim of this work is to establish the asymptotic behavior of solutions to the following heat equation in
nonhomogeneous media with critical density:
\begin{equation}\label{eq1}
|x|^{-2}\partial_{t}u(x,t)=\Delta u(x,t), \quad
(x,t)\in\real^N\times(0,\infty).
\end{equation}
as a part of an ongoing project of studying the asymptotic behavior
for
\begin{equation}\label{eq11}
|x|^{-2}\partial_{t}u(x,t)=\Delta u^m(x,t), \quad
(x,t)\in\real^N\times(0,\infty),
\end{equation}
with $m\geq1$. The technically more involved problem with the
large-time behavior for $m>1$ will be studied in a forthcoming paper
\cite{IS12}.

Equations of type \eqref{eq11} with general densities have been
proposed by Kamin and Rosenau in a series of papers \cite{KR81,
KR82, KR83} to model thermal propagation by radiation in
non-homogeneous plasma. Since then, many papers were devoted to
developing rigorously the qualitative theory or asymptotic behavior
for
\begin{equation}\label{eq.general}
\varrho(x)\partial_{t}u(x,t)=\Delta u^m(x,t), \quad
(x,t)\in\real^N\times(0,\infty),
\end{equation}
usually asking that $\varrho(x)\sim|x|^{-\gamma}$ as $|x|\to\infty$,
for some $\gamma>0$ (e.g. \cite{RV08}). Thus, for $m>1$, it has been
noticed that while $0<\gamma<2$, the solutions have similar
properties to the ones of the pure porous medium equation (for
short, PME)
$$
\partial_{t}u=\Delta u^m,
$$
see \cite{RV06, RV09}, while for $\gamma>2$, the properties of the
solutions depart strongly from the ones of the PME \cite{KRV10}.
Thus, $\gamma=2$ is critical, and the asymptotic behavior for this
case is left open in \cite{KRV10} with a conjecture giving the
explicit profile. On the other hand, the asymptotic behavior for
\eqref{eq.general} with a density $\varrho(x)\sim|x|^{-2}$ at
infinity, but $\varrho$ regular near the origin, is studied in
another recent work \cite{NR}, obtaining an explicit profile solving
\eqref{eq11}, and proving the convergence towards it in the outer
region (that is, outside small compacts near $x=0$). Unusually, the
linear diffusion problem $m=1$ has been studied later than its
nonlinear version, see for example \cite{EKP00, EK05}.

In all cases, it has been noticed that the solutions converge
asymptotically towards profiles coming from the pure power density
equation, that is, $\varrho(x)=|x|^{-\gamma}$. Moreover,
\begin{equation}\label{eq2}
|x|^{-\gamma}\partial_{t}u(x,t)=\Delta u^{m}(x,t)
\end{equation}
has some more interesting feature: a singularity at $x=0$, apart
from the decay at infinity. Recently, in \cite{IRS} the authors
study formally some properties of radial solutions to \eqref{eq2} as
a first step to understand its general behavior. Moreover, a study
of existence and uniqueness for \eqref{eq2} for $\gamma>2$ and $m>1$
is done in \cite[Section 6]{KRV10}.

Coming back to our problem, that of letting $m=1$, $\gamma=2$ in
\eqref{eq2}, we find explicit asymptotic profiles which explain
better the effect of the singularity at $x=0$. The general case
$m>1$ will be treated in the companion paper \cite{IS12}. But in
order to explain these comments and to make precise the motivation
for this work, let us state the main results of the paper.

\medskip

\noindent \textbf{Main results.} We deal with the Cauchy problem
associated to Eq. \eqref{eq1} with initial data
\begin{equation}\label{initdata}
u_0\in L^1_2(\real^N), \quad u_0\geq0,
\end{equation}
where the dimension is $N\geq3$ (except when specified) and
$$
L^1_{2}(\real^N):=\left\{h:\real^N\mapsto\real, \ h \ {\rm
measurable}, \int_{\real^N}|x|^{-2}h(x)\,dx<\infty \right\}.
$$
In Section \ref{sec.wp} we give the precise notions of \emph{weak
solution} and \emph{strong solution} to \eqref{eq1} and we prove
that the Cauchy problem for \eqref{eq1} is well-posed in our
framework. We refer the interested reader to Definition
\ref{def.weak} and Theorem \ref{th.ex} for the precise statements.

We state the results about the large-time behavior, that we find
quite interesting and unexpected. We begin with the case when
$u_0(0)=0$.
\begin{theorem}\label{th.1}
Let $u$ be a \textbf{radially symmetric} solution of Eq. \eqref{eq1}
with initial data satisfying \eqref{initdata} and moreover
\begin{equation}\label{initdata2}
M_{u_0}:=\int_{\real^N}|x|^{-N}u_0(x)\,dx<\infty, \quad u_0(0)=0.
\end{equation}
Then we have
\begin{equation}\label{asympt1}
\lim\limits_{t\to\infty}t^{1/2}\left\|u(x,t)-\frac{M_{u_0}}{\omega_1}F(x,t)\right\|_{\infty}=0,
\end{equation}
where $\omega_1$ is the area of the unit sphere in $\real^N$ and
\begin{equation}\label{profile1}
F(x,t):=\left\{\begin{array}{ll}\frac{1}{\sqrt{4\pi t}}G\left(\frac{\log|x|+(N-2)t}{2\sqrt{t}}\right), \quad
G(\xi)=e^{-\xi^2}, \quad {\rm for} \ |x|\neq0, \\0, \quad {\rm for} \ x=0,\end{array}\right.
\end{equation}
\end{theorem}
\noindent \textbf{Remarks.} (a) Let us notice also that
$$
\max\{F(x,t): x\neq0\}=O(t^{-1/2}), \quad {\rm as} \ t\to\infty,
$$
showing that the time-scale $t^{1/2}$ is the correct one for the asymptotic behavior in \eqref{asympt1}. More
precisely,
\begin{equation}\label{cex}
\|F(\cdot,t)\|_{\infty}=\frac{1}{\sqrt{4\pi t}},
\end{equation}
as it will be analyzed in the remarks at the end of Section \ref{sec.asympt1}.

\noindent (b) The mass $M_{u_0}$ in \eqref{initdata2} is conserved along the flow. This will be obvious from the
proof of Theorem \ref{th.1}.

\medskip

\noindent \textbf{Counterexamples in the non-radial case.} We
emphasize on the fact that this is true \emph{only for radially
symmetric solutions}. For general solutions, Theorem \ref{th.1} is
not true. It is quite surprising that, due to the singularity at
$x=0$, the solutions that start from $u_0(0)=0$ do not converge
asymptotically to a radial profile.

We construct a counterexample to Theorem \ref{th.1} for solutions
that are not radially symmetric. We pass to generalized spherical
coordinates $x=(r,\phi_1,...,\phi_{N-2},\theta)$ in $\real^N$ and
define the following function
$$
F_{N}(x,t)=F_N(r,\phi_1,...,\phi_{N-2},\theta):=\theta F(r,t).
$$
Using the formula of the Laplace operator in spherical coordinates,
from the particular form of $F_N$ (that involves no dependence on
$(\phi_1,...,\phi_{N-2})$ and $\partial_{\theta}^2F_{N}\equiv0$), we
find that $F_N$ is a solution to \eqref{eq1}, since
$$
\Delta
F_N(x,t)=\theta\Delta_{r}F(r,t)=\theta|x|^{-2}F_t(r,t)=|x|^{-2}\partial_{t}F_{N}(x,t),
$$
where $\Delta_r$ is the Laplacian operator in radial variable. The
fact that Theorem \ref{th.1} is not satisfied on this example is
obvious since \eqref{asympt1} becomes
$$
\lim\limits_{t\to\infty}t^{1/2}\|(\theta-K)F(r,t)\|_{\infty}=0,
$$
for some constant $K$ depending on the dimension $N$, which is false
since $\theta$ is variable, $K$ is constant and $t^{1/2}\|F(\cdot,
t)\|_{\infty}=1/\sqrt{4\pi}$.

\medskip

When the value of the initial data at the origin is nonzero, things
are completely different, as the following theorem states.
\begin{theorem}\label{th.2}
Let $u$ be a \textbf{general} solution of Eq. \eqref{eq1} with initial data $u_0$ satisfying \eqref{initdata}
and $u_0(0)=K>0$. If there exist $\delta$, $\e>0$ such that the initial data $u_0$ satisfies
\begin{equation}\label{initdata1}
|K-u_0(x)|\leq|x|^{\delta}, \quad {\rm as} \ |x|\to0, \quad u_0(x)\leq |x|^{-\e}, \quad {\rm as} \ |x|\to\infty,
\end{equation}
then we have
\begin{equation}\label{asympt0}
\left\|u(t)-\frac{K}{2}E(t)\right\|_{\infty}=O(t^{-1/2}), \quad {\rm
as} \ t\to\infty,
\end{equation}
where
\begin{equation}\label{profile2}
E(x,t):=\left\{\begin{array}{ll}{\rm erfc}\left(\frac{\log|x|+(N-2)t}{2\sqrt{t}}\right), \quad {\rm
erfc}(\xi)=\frac{2}{\sqrt{\pi}}\int_{\xi}^{\infty}e^{-\theta^2}\,d\theta, \quad {\rm for} \ |x|\neq0,\\
K, \quad {\rm for} \ x=0.\end{array}\right.
\end{equation}
\end{theorem}
\noindent \textbf{Remarks. 1.} Theorem \ref{th.2} can be improved by relaxing and generalizing the condition
\eqref{initdata1}; indeed, \eqref{asympt0} holds true if there exists some functions $\Phi_1,
\Phi_2:(0,\infty)\mapsto(0,\infty)$ such that
$$
\int_{0+}\frac{\Phi_1(r)}{r}\,dr<\infty, \quad \int^{+\infty}\frac{1}{r\Phi_2(r)}\,dr<\infty,
$$
and $$u_0(x)\leq\Phi_1(|x|), \quad \hbox{as} \ x\to0, \quad u_0(x)\leq\Phi_2(|x|), \quad \hbox{as} \
|x|\to\infty.$$ The proof is totally similar.

\noindent \textbf{2. } Theorem \ref{th.2} is indeed true in general, also without radial symmetry, in contrast
with Theorem \ref{th.1}. These two different results show the importance of the singularity of the density
$|x|^{-2}$. Indeed, the evolution gives rise in a striking way to two different asymptotic profiles, only
depending on a difference in the pointwise value of $u_0$ at the origin.

When dealing with radially symmetric solutions, the results can be improved both with respect to the condition
on the initial data $u_0$ and with respect to the rate of convergence. This is gathered in the following
\begin{theorem}\label{th.3}
Let $u$ be a \textbf{radially symmetric} solution for Eq. \eqref{eq1} whose initial data satisfies
\eqref{initdata}.

\noindent (a) The asymptotic convergence \eqref{asympt0} holds true
under the following condition on $u_0$:
\begin{equation}\label{initdata1bis}
I_1:=\int_{B(0,1)}|x|^{-N}|K-u_0(x)|\,dx+\int_{\real^N\setminus B(0,1)}|x|^{-N}|u_0(x)|\,dx<\infty.
\end{equation}

\noindent (b) If we ask furthermore the following condition on the gradient
\begin{equation}\label{initdata3}
I_2:=\int_{\real^N}\frac{|(\log|x|)^3|}{|x|^{N-1}}|\nabla
u_0(x)|\,dx<\infty,
\end{equation}
then we have a rate of convergence:
\begin{equation}\label{asympt2}
\left\|u(t)-\frac{K}{2}E(t)\right\|_{\infty}=O(t^{-3/2}), \quad {\rm
as} \ t\to\infty.
\end{equation}
\end{theorem}

We will add here two more interesting consequences about the behavior near $x=0$, where the singularity at $x=0$
plays a key role.
\begin{proposition}\label{prop.3}
\noindent (a) There is no fundamental solution of Eq. \eqref{eq1},
that is, a solution with initial data $u_0=M\delta_0$.

\noindent (b) If the initial condition $u_0$ satisfies $u_0(x)=0$ for any $x\in B(0,r_0)$, for some $r_0>0$,
then $u(x,t)>0$ for all $(x,t)\in(\real^N\setminus\{0\})\times(0,\infty)$ and $u(0,t)=0$ for any $t>0$.
\end{proposition}

We plot the asymptotic profiles in Figure \ref{figure1} in dimension
$N=3$, the general picture being similar. In the pictures we see the
profiles evolving at various times.

\begin{figure}[ht!]
  \begin{center}
  \includegraphics[width=15cm,height=10cm]{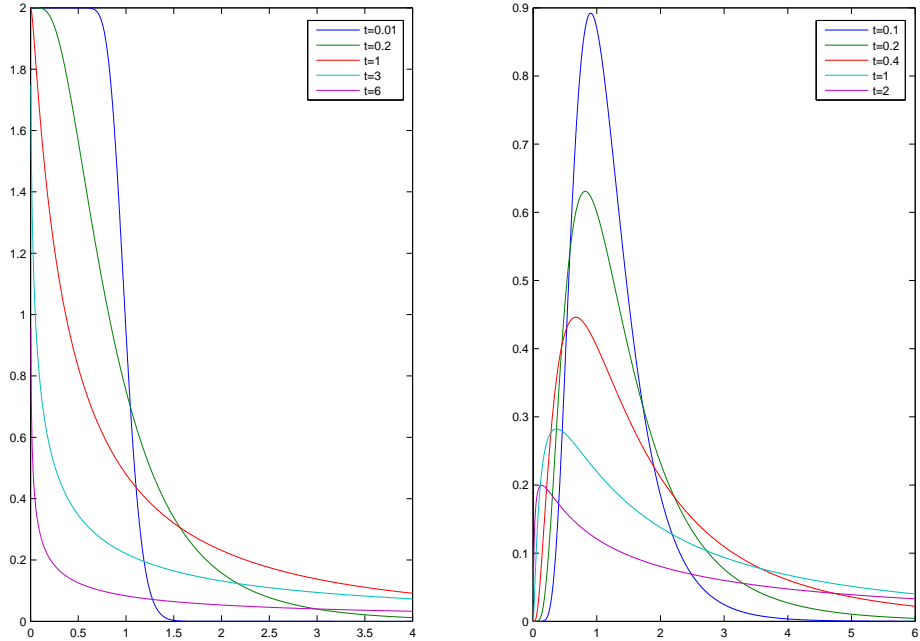}
  \end{center}
  \caption{Profiles $E$ and $F$ for dimension $N=3$ at various times.} \label{figure1}
\end{figure}

\medskip

\noindent \textbf{Organisation of the paper.} Before passing to the
asymptotic behavior, we begin by studying the well-posedness for
\eqref{eq1} with suitable initial data. This is the goal of Section
\ref{sec.wp}, which insures us that the object of our study exists
in suitable spaces. We then prove the main results in two steps. In
a first step, we describe a transformation mapping radially
symmetric solutions of \eqref{eq1} into solutions of the
one-dimensional heat equation. This is done in Section \ref{sec.tr}.
With the aid of it, we prove Theorem \ref{th.1} and part (a) of
Theorem \ref{th.3}. Then, in order to prove part (b) of Theorem
\ref{th.3}, we need one more transformation step. The proofs of the
theorems for radially symmetric solutions and some more remarks
about the asymptotic profiles are the subject of Section
\ref{sec.asympt1}. Finally, in Section \ref{sec.asympt2} we prove
Theorem \ref{th.2}, using as essential tool a comparison principle
proved in Section \ref{sec.wp}, and we end with the proof of
Proposition \ref{prop.3}. We close the paper with a section of
comments where we include a brief discussion of the dimension $N=1$
for \eqref{eq1} and raise some open problems.

\section{Existence and uniqueness}\label{sec.wp}

Before proving the main results about asymptotic behavior, we have
to develop a theory of existence and uniqueness for the Cauchy
problem
\begin{equation}\label{CP1}
\left\{\begin{array}{ll}|x|^{-2}u_t=\Delta u, \quad \hbox{in} \
\real^N\times(0,\infty),\\ u(x,0)=u_0(x), \quad
x\in\real^N,\end{array}\right.
\end{equation}
for suitable initial data $u_0\in L^1_2(\real^N)$. We start from the
precise definition of a solution.
\begin{definition}\label{def.weak}
We say that a function $u(x,t)$ is a \emph{weak solution} to
\eqref{CP1} if it satisfies the following conditions:

\noindent (a) $u\in C([0,\infty);L^1_2(\real^N))\cap
L^{\infty}(\real^N\times(\tau,\infty))$, for any $\tau>0$, and
$u\geq0$ in $\real^N\times(0,\infty)$;

\noindent (b) The integral identity
\begin{equation}\label{weaksol}\int\int
u(\Delta\varphi+|x|^{-2}\varphi_t)\,dx\,dt=0\end{equation} holds
true for any test function $\varphi\in
C_0^{\infty}(\real^N\times(0,\infty))$, and $u(0)=u_0$.
\end{definition}

We have the following:
\begin{theorem}\label{th.ex}
Let $u_0\in L^{1}_{2}(\real^N)$, $u_0$ nonnegative. Then there
exists a unique weak solution $u$ to the Cauchy problem \eqref{CP1}
such that
$$
|x|^{-2}u_t, \ \Delta u\in L^1_{\rm loc}(Q_*), \quad
|x|^{-2}u_t=\Delta u \ a.\,e. \ {\rm in} \ Q_*,
$$
where $Q_*=\real^N\times(0,\infty)\setminus\{(0,t):t>0\}$.
\end{theorem}
This type of solution is usually referred as a \emph{strong
solution} (see \cite{VazquezPME}). The proof will adapt ideas from
the proof of existence in \cite[Section 6]{KRV10}, where
well-posedness is proved for \eqref{eq2} with $\gamma>2$ and $m>1$,
thus at some points we will be rather sketchy.
\begin{proof}[Proof of Theorem \ref{th.ex}]
\textbf{Uniqueness.} This is based on the following result which is
also interesting by itself.
\begin{proposition}[$L^{1}_2$-Contraction principle]\label{prop.contraction}
Let $u_1$, $u_2$ be two strong solutions of Eq. \eqref{eq1}. For
$0<t_1<t_2$ we have
\begin{equation}\label{contr}
\int_{\real^N}|x|^{-2}\left[u_1(x,t_2)-u_2(x,t_2)\right]_{+}\,dx\leq\int_{\real^N}|x|^{-2}\left[u_1(x,t_1)-u_2(x,t_1)\right]_{+}\,dx,
\end{equation}
where $[g]_{+}$ represents the positive part.
\end{proposition}
\begin{proof}
We follow the same ideas as in \cite[Proposition 9.1]{VazquezPME}.
Let $p\in C^1(\real)\cap L^{\infty}(\real)$ be such that $p(s)=0$
for $s\leq0$, $p'(s)>0$ for $s>0$ and $0\leq p\leq1$. Let $j$ be the
primitive of $p$ such that $j(0)=0$. The idea is to choose $p$ as an
approximation of the function
\begin{equation*}
\hbox{sign}_0^{+}(r)=\left\{\begin{array}{ll}1, \quad \hbox{if} \
r>0,
\\ 0, \quad \hbox{if} \ r\leq0,\end{array}\right.
\end{equation*}
thus $j$ will approximate the positive part function. Consider also
a cutoff function $\xi_n$ constructed in the following way: let
$\xi_0\in C_{0}^{\infty}(\real^N)$, $0\leq\xi_0\leq1$, $\xi_0(x)=0$
for $|x|\geq2$, $\xi_0(x)=1$ for $|x|\leq1$, and define
$\xi_n(x):=\xi_0(x/n)$. Subtracting the equations satisfied by $u_2$
and $u_1$, then multiplying by the test function $p(u_1-u_2)\xi_n$,
we obtain
\begin{equation}\label{interm3}
\int_{t_1}^{t_2}\int_{\real^N}|x|^{-2}(u_1-u_2)_{t}p(u_1-u_2)\xi_n\,dx\,dt=\int_{t_1}^{t_2}\int_{\real^N}\Delta(u_1-u_2)p(u_1-u_2)\xi_n\,dx\,dt.
\end{equation}
By approximation of $u_1-u_2$ as in the proof of \cite[Proposition
9.1]{VazquezPME} and integration by parts in the right-hand side of
\eqref{interm3}, we obtain, with $\xi=\xi_n$:
\begin{equation*}
\begin{split}
\int_{t_1}^{t_2}\int_{\real^N}|x|^{-2}(u_1-u_2)_tp(u_1-u_2)\xi\,dx\,dt&\leq-\int_{t_1}^{t_2}\int_{\real^N}p(u_1-u_2)\nabla(u_1-u_2)\cdot\nabla\xi\,dx\,dt\\
&=-\int_{t_1}^{t_2}\int_{\real^N}\nabla
j(u_1-u_2)\cdot\nabla\xi\,dx\,dt\\&=\int_{t_1}^{t_2}\int_{\real^N}j(u_1-u_2)\Delta\xi\,dx\,dt\leq\int_{t_1}^{t_2}\int_{\real^N}|u_1-u_2||\Delta\xi|\,dx\,dt.
\end{split}
\end{equation*}
We let now $p\to\hbox{sign}_0^{+}$ and integrate to get
\begin{equation*}
\begin{split}
\int_{\real^N}|x|^{-2}\left[u_1(x,t_2)-u_2(x,t_2)\right]_{+}\xi_n\,dx&\leq\int_{\real^N}|x|^{-2}\left[u_1(x,t_1)-u_2(x,t_1)\right]_{+}\xi_n\,dx\\&+
\|\Delta\xi_{n}\|_{\infty}\int_{t_1}^{t_2}\int_{|x|>n}|u_1-u_2|\,dx\,dt.
\end{split}
\end{equation*}
Letting now $n\to\infty$ and taking into account that we deal with
solutions in $L^1$, we obtain the result.
\end{proof}
The uniqueness follows obviously from the contraction principle.
Moreover, we have the following:
\begin{corollary}[Comparison principle]\label{cor.comp}
If $u_1$, $u_2$ are solutions of Eq. \eqref{eq1} such that their
initial data satisfy $u_{0,1}\leq u_{0,2}$, then $u_1\leq u_2$ in
$\real^N\times(0,\infty)$.
\end{corollary}
\begin{proof}[Proof of Corollary \ref{cor.comp}]
Suppose that $u_{0,1}$, $u_{0,2}\in L^1_2(\real^N)$ satisfy
$u_{0,1}\leq u_{0,2}$ and $u_1$, $u_2$ are the corresponding
solutions. Then, for any $t>0$, we have
\begin{equation*}
\int_{\real^N}|x|^{-2}\left[u_1(x,t)-u_2(x,t)\right]_{+}\,dx\leq\int_{\real^N}|x|^{-2}\left[u_{0,1}(x)-u_{0,2}(x)\right]_{+}\,dx=0,
\end{equation*}
hence $u_1(x,t)\leq u_2(x,t)$.
\end{proof}

\medskip

\noindent \textbf{Existence.} This part of the proof is more
involved and will be divided, for a better comprehension, into
several steps.

\medskip

\noindent \textbf{Step 1. Compactly supported data.} In a first
step, we prove existence of a strong solution for compactly
supported initial data outside the origin. For $0<r<R<\infty$, we
denote an annulus by $B_{r,R}:=B_R\setminus B_{r}$. Given $u_0\in
C_{0}^{\infty}(\real^N\setminus\{0\})$, consider the approximating
Dirichlet problem
\begin{equation}\label{DP1}
\left\{\begin{array}{ll}|x|^{-2}u_t=\Delta u, \quad \hbox{in} \
Q_{r,R}:=B_{r,R}\times(0,\infty),\\ u(x,0)=u_0(x), \quad \hbox{in} \
B_{r,R},\\ u(x,t)=0, \quad \hbox{in} \ \partial
B_{r,R}\times(0,\infty).\end{array}\right.
\end{equation}
The problem \eqref{DP1}, being nor singular at $x=0$ neither
degenerate at infinity, admits a unique solution $u_{r,R}$ as shown
in \cite{Ei90, EK94}. Moreover, by simple comparison in the smaller
annulus, it is easily seen that $u_{r_2,R_2}\geq u_{r_1,R_1}$ if
$r_2\leq r_1$ and $R_2\geq R_1$. It remains to show that they are
uniformly bounded. To this end, recall that $u_0\in
C_{0}^{\infty}(\real^N\setminus\{0\})$ and consider $r>0$ small
enough and $R>0$ large enough such that $\hbox{supp}\,u_0\subset
B_{r,R}$. Recalling the definition of $F(x,t)$ given in
\eqref{profile1}, for some $\tau>0$ sufficiently small and $K>0$
large, we have
\begin{equation*}
u_0(x)\leq KF(x,\tau) \quad \hbox{in} \ B_{r,R}.
\end{equation*}
Since on the lateral boundary $u_{r,R}(x,t)=0\leq KF(x,t+\tau)$ for
any $x$ such that either $|x|=r$ or $|x|=R$, by standard comparison
we obtain the following important universal bound
\begin{equation}\label{unif.bound}
u_{r,R}(x,t)\leq KF(x,t+\tau), \quad \hbox{in} \
B_{r,R}\times(0,\infty).
\end{equation}
This is an uniform bound which does not depend on $r$, $R$. Thus,
the following limit
\begin{equation}\label{sol}
u(x,t)=\lim\limits_{r\to0}\lim\limits_{R\to\infty}u_{r,R}(x,t),
\end{equation}
is well-defined and gives rise to a weak solution to \eqref{eq1}, as
it is easy to check with the definition. In this way, we have solved
the problem for $u_0\in C_0^{\infty}(\real^N\setminus\{0\})$.
Moreover, let us notice here that any solution obtained through this
approximation process also satisfies the universal bound
\eqref{unif.bound} in $Q_*$.

\medskip

\noindent \textbf{Step 2. Preliminary estimates.} We still want to
prove that our weak solutions are in fact strong, then pass to the
general case by some approximation process. To this end, we
establish two estimates that will imply further regularity on the
solutions. They are gathered in the following
\begin{lemma}\label{lem.estimate}
Let $u$ be a weak solution to \eqref{eq1} such that $u_0\in
C_0^{\infty}(\real^N\setminus\{0\})$. Let
$K\subset\real^N\setminus\{0\}$ be a compact set and
$0<\tau<T<\infty$. Then there exists a positive constant $C>0$
depending only on $N$, ${\rm dist}(\{0\},K)$ and $\tau$ such that
\begin{equation}\label{est1}
\int_{\tau}^{T}\int_{K}|\nabla u|^2\,dx\,dt\leq C
\end{equation}
and
\begin{equation}\label{est2}
\int_{\tau}^{T}\int_{K}|x|^{-2}|u_t|^2\,dx\,dt\leq C.
\end{equation}
\end{lemma}
\begin{proof}
The proof adapts the ones in \cite[Theorem 3.6, Theorem 3.8]{KRV10},
but we give a complete version of it for the sake of completeness
and simplicity. Let $\overline{K}$ be a compact neighborhood of $K$
such that $0$ does not belong to $\overline{K}$ and consider a
cutoff function $\eta\in C_{0}^{\infty}(\real^N)$ such that
$$
\eta\equiv1 \ \hbox{in} \ K, \ \eta\equiv0 \ \hbox{in} \
\real^N\setminus\overline{K}, \quad 0\leq\eta\leq1 \ \hbox{in} \
\real^N.
$$
We multiply \eqref{eq1} by $u\eta^2$ and integrate by parts to get
\begin{equation*}
\begin{split}
\frac{d}{dt}\int_{\overline{K}}|x|^{-2}\frac{u^2}{2}\eta^2\,dx&=\int_{\overline{K}}u\Delta
u\eta^2\,dx=-\int_{\overline{K}}\nabla u\cdot\nabla(u\eta^2)\,dx\\
&=-\int_{\overline{K}}|\nabla
u|^2\eta^2\,dx-2\int_{\overline{K}}u\eta\nabla u\cdot\nabla\eta\,dx.
\end{split}
\end{equation*}
Integrating on $[\tau,T]$ and using Young's inequality, we further
obtain
\begin{equation*}
\begin{split}
\frac{1}{2}\int_{\overline{K}}|x|^{-2}u^{2}(x,T)\eta^2\,dx&+\int_{\tau}^{T}\int_{\overline{K}}|\nabla
u|^2\eta^2\,dx\,dt=\frac{1}{2}\int_{\overline{K}}|x|^{-2}u^{2}(x,\tau)\eta^2\,dx-2\int_{\tau}^{T}\int_{\overline{K}}u\eta\nabla
u\cdot\nabla\eta\,dx\,dt\\&\leq\frac{1}{2}\int_{\overline{K}}|x|^{-2}u^{2}(x,\tau)\eta^2\,dx+\frac{1}{2}\int_{\tau}^{T}\int_{\overline{K}}|\nabla
u|^2\eta^2\,dx\,dt+2\int_{\tau}^{T}\int_{\overline{K}}u^2|\nabla\eta|^2\,dx\,dt,
\end{split}
\end{equation*}
whence estimate \eqref{est1} follows from the uniform bound
\eqref{unif.bound}, which implies readily an uniform bound for the
terms in the right-hand side. Consequently,
$$
\int_{\tau}^{T}\int_{K}|\nabla
u|^2\,dx\,dt\leq\int_{\tau}^{T}\int_{\overline{K}}|\nabla
u|^2\eta^2\,dx\,dt\leq C$$ as desired.

In order to prove the second estimate, we multiply \eqref{eq1} this
time by $u_{t}\eta$, then integrate in space and time. We obtain
\begin{equation*}
\int_{\overline{K}}|x|^{-2}|u_t|^2\eta\,dx=-\int_{\overline{K}}\nabla
u\cdot\nabla(u_{t}\eta)\,dx=-\frac{1}{2}\frac{d}{dt}\int_{\overline{K}}|\nabla
u|^2\eta\,dx,
\end{equation*}
whence
\begin{equation*}
\begin{split}
\int_{\tau}^{T}\int_{K}|x|^{-2}|u_t|^2\,dx\,dt&\leq\int_{\tau}^{T}\int_{\overline{K}}|x|^{-2}|u_t|^2\eta\,dx\,dt\\
&=\frac{1}{2}\int_{\overline{K}}|\nabla
u(x,\tau)|^2\eta\,dx-\frac{1}{2}\int_{\overline{K}}|\nabla
u(x,T)|^2\eta\,dx\leq C,
\end{split}
\end{equation*}
the last step coming from the first estimate \eqref{est1}.
\end{proof}
In particular, we deduce from Lemma \ref{lem.estimate} that
$|x|^{-1}u_t=|x|\Delta u\in L^2_{{\rm
loc}}(\real^N\times(0,\infty))$, which at its turn shows that
$|x|^{-2}u_t$, $\Delta u\in L^1_{{\rm loc}}(Q_*)$. Hence $u$ is a
strong solution.

\medskip

\noindent \textbf{Step 3. General data.} For general data $u_0\in
L^1_2(\real^N)$, we use the contraction principle \eqref{contr}. Let
$u_{0,n}$ be a sequence of functions in
$C_0^{\infty}(\real^N\setminus\{0\})$ such that $u_{0,n}\to u_{0}$
in $L_2^1(\real^N)$ and $u_n$ the corresponding solution with
initial datum $u_{0,n}$. By \eqref{contr}, the sequence $u_{n}$ is
uniformly Cauchy in $C([0,\infty);L^1_2(\real^N))$, hence it
converges to a limit $u\in C([0,\infty);L^1_2(\real^N))$ which is a
weak solution (passing to the limit in the identity \eqref{weaksol}
is immediate). Moreover, we deduce from \eqref{unif.bound} that $u$
is uniformly bounded on compact subsets of $Q_*$. On the other hand,
we apply Lemma \ref{lem.estimate} for $u_n$ and we obtain the
estimates \eqref{est1}, \eqref{est2} locally uniformly with respect
to $n$ in $Q_*$. It follows first than the limit $u$ satisfies the
bound \eqref{unif.bound}, then \eqref{est1}, and finally, by redoing
for $u$ the last step in the proof of Lemma \ref{lem.estimate}, we
deduce that $u$ also satisfies \eqref{est2}. Hence $u$ is a strong
solution.
\end{proof}

\section{Radially symmetric solutions. The
transformations}\label{sec.tr}

We restrict ourselves first to radially symmetric solutions
$u(r,t)=u(|x|,t)$, $r=|x|$, of Eq. \eqref{eq1}. In this case, we
introduce the following change of variable
\begin{equation}\label{tr1}
u(|x|,t)=v(y,t), \quad y=\log|x|+(N-2)t.
\end{equation}
By a simple calculation, we notice that
$$
u_t(r,t)=(N-2)v_{y}(y,t)+v_t(y,t), \quad
u_r(r,t)=\frac{1}{r}v_{y}(y,t), \quad
u_{rr}(r,t)=\frac{1}{r^2}(v_{yy}(y,t)-v_{y}(y,t)),
$$
hence, by replacing in \eqref{eq1}, we obtain that $v$ is a solution
to the Cauchy problem for the one-dimensional heat equation
\begin{equation}\label{1DHeat}
v_{t}(y,t)=v_{yy}(y,t), \quad v_0(y)=u_0(r).
\end{equation}
for $(y,t)\in\real\times(0,\infty)$. We observe that $x=0$ is mapped
through \eqref{tr1} into $y\to-\infty$. We thus can use the previous
transformation \eqref{tr1} in order to prove Theorems \ref{th.1} and
\ref{th.2}. But before doing that, let us also notice that, in
dimensions $N\geq3$ and also $N=1$, there exists another
transformation, appearing in a more general version in
\cite[Subsection 1.3.4]{Pol},
\begin{equation}\label{tr2}
u(r,t)=e^{z-t}w(z,t), \quad z=-\frac{N-2}{2}\log|x|,
\end{equation}
again mapping radially symmetric solutions $u(r,t)$ of Eq.
\eqref{eq1} into solutions $w(z,t)$ of \eqref{1DHeat}, but with a
different connection between initial data:
$$
u_0(|x|)=e^{z}w_0(z)=|x|^{-(N-2)/2}w_0(z).
$$
The change of variable and function \eqref{tr2} acts as an inversion for $N\geq3$: it maps $x=0$ to
$y\to+\infty$ and $|x|\to+\infty$ to $y\to-\infty$, and it acts as a direct mapping only for $N=1$. In the
sequel, we will use mainly transformation \eqref{tr1} due to its simplicity, but at some points \eqref{tr2} will
appear too.

\section{Radially symmetric solutions. Asymptotic
convergence}\label{sec.asympt1}

In this section we prove Theorems \ref{th.1} and \ref{th.2} for
radially symmetric solutions and we deduce some more interesting
remarks about the quite unexpected asymptotic behavior for Eq.
\eqref{eq1}.

\begin{proof}[Proof of Theorem \ref{th.1}] Let $u$ be a solution for
\eqref{eq1} with initial datum satisfying \eqref{initdata} and
\eqref{initdata2}. By transformation \eqref{tr1}, we arrive to a
solution $v$ of \eqref{1DHeat} with initial datum $v_0$ such that
$$\lim\limits_{y\to\infty}v_0(y)=\lim\limits_{y\to-\infty}v_0(y)=0$$ and
\begin{equation}\label{mass}
\begin{split}
M_{v_0}:&=\int_{-\infty}^{\infty}v_0(y)\,dy=\int_{-\infty}^{\infty}u_0(e^y)\,dy=\int_0^{\infty}\frac{u_0(r)}{r}\,dr\\
&=\int_0^{\infty}r^{-N}u_0(r)r^{N-1}\,dr=\frac{1}{\omega_1}\int_{\real^N}|x|^{-N}u_0(|x|)\,dx=\frac{M_{u_0}}{\omega_1}<\infty,
\end{split}
\end{equation}
where, as usual, $\omega_1$ is the area of the unit sphere of $\real^N$. In fact, we notice that the same
equality \eqref{mass} holds true taken at any time $t>0$ instead of $t=0$. From the standard mass conservation
property for the heat equation, we deduce that the quantity $\int_{\real^N}|x|^{-N}u(x,t)\,dx$ is conserved.

Due to well-known results in the theory of the heat equation, we find
\begin{equation}\label{interm1}
\lim\limits_{t\to\infty}t^{1/2}\left|v(y,t)-M_{v_0}\frac{1}{\sqrt{4\pi
t}}e^{-y^2/4t}\right|=0,
\end{equation}
uniformly for $y\in\real$. We translate \eqref{interm1} in terms of
$u$ to get
\begin{equation}\label{asympt1exp}
\lim\limits_{t\to\infty}t^{1/2}\left|u(|x|,t)-\frac{M_{u_0}}{\omega_1}\frac{1}{\sqrt{4\pi
t}}\exp\left(-\frac{(\log|x|+(N-2)t)^2}{4t}\right)\right|=0,
\end{equation}
that is \eqref{asympt1}.
\end{proof}
We prove now the asymptotic convergence for the case when $u_0(0)=K>0$.

\begin{proof}[Proof of Theorem \ref{th.3}] (a) Let $u$ be a solution for \eqref{eq1} with initial datum satisfying $u_0(0)=K>0$, \eqref{initdata} and
\eqref{initdata1bis}. By the transformation \eqref{tr1}, we arrive to a solution $v$ of \eqref{1DHeat} whose
initial datum $v_0$ satisfies
\begin{equation}\label{interm4}
\lim\limits_{y\to\infty}v_0(y)=0, \quad
\lim\limits_{y\to-\infty}v_0(y)=K>0.
\end{equation}
Moreover, we notice that
\begin{equation}\label{interm6}
\begin{split}
\int_{-\infty}^{\infty}|KH(y)-v_0(y)|\,dy&=\int_{-\infty}^0|K-v_0(y)|\,dy+\int_0^{\infty}|v_0(y)|\,dy\\
&=\int_0^1\frac{|K-u_0(r)|}{r}\,dr+\int_1^{\infty}\frac{u_0(r)}{r}\,dr\leq\frac{1}{\omega_1}I_1<\infty,
\end{split}
\end{equation}
due to the conditions on $u_0$ in \eqref{initdata1bis}, where we have denoted by $H$ the complementary Heaviside
function
$$
H(y)=\left\{\begin{array}{ll}1, \quad \hbox{for} \ y<0,\\0, \quad
\hbox{for} \ y\geq0.\end{array}\right.
$$
We now use the following
\begin{lemma}\label{lem.Heat}
Let $v_1$, $v_2$ be two solutions of the one-dimensional heat
equation with initial data satisfying
\begin{equation}\label{Hcond}
\int_{-\infty}^{\infty}|v_2(y,0)-v_1(y,0)|\,dy<\infty.
\end{equation}
Then we have
\begin{equation}\label{interm5}
\|v_1(t)-v_2(t)\|_{\infty}=O(t^{-1/2}), \quad {\rm as} \ t\to\infty.
\end{equation}
\end{lemma}
The proof of Lemma \ref{lem.Heat} is immediate, since
$w(y,t):=v_2(y,t)-v_1(y,t)$ is a solution for \eqref{1DHeat} with
integrable initial datum, thus has the desired order as $t\to\infty$
by standard results for the heat equation.

We apply now Lemma \ref{lem.Heat} for the solutions $v(y,t)$ and
$K\hbox{erfc}(y,t)/2$ of \eqref{1DHeat}. Their initial data are
$v_0(y)$ and $KH(y)$ in the previous notations, thus they satisfy
the condition \eqref{Hcond}, due to \eqref{interm6} and the fact
that $I_1<\infty$. Thus, we get that
\begin{equation*}
\left\|v(t)-\frac{K}{2}\hbox{erfc}(t)\right\|_{\infty}=O(t^{-1/2}),
\quad {\rm as} \ t\to\infty,
\end{equation*}
whence, undoing the transformation \eqref{tr1} and coming back to
the initial variables, we deduce \eqref{asympt0}.

(b) Let $u$ be a solution for \eqref{eq1} with initial datum satisfying $u_0(0)=K>0$, \eqref{initdata},
\eqref{initdata1bis} and \eqref{initdata3}. By the transformation \eqref{tr1}, we arrive to a solution $v$ of
\eqref{1DHeat} with initial datum $v_0$ satisfying \eqref{interm4}. We make a further change by letting
$w(y,t):=K-v(y,t)$, hence $w$ is a solution to \eqref{1DHeat} with
$$
\lim\limits_{y\to\infty}w_0(y)=K>0, \quad
\lim\limits_{y\to-\infty}w_0(y)=0.
$$
Let $\psi(y,t):=w_{y}(y,t)=-v_{y}(y,t)$. Then again $\psi$ is a
solution for \eqref{1DHeat}, and we want to apply for $\psi$ the
following result:
\begin{lemma}\label{lem.Miller}
Let $\psi$ be a solution for Eq. \eqref{1DHeat}, whose initial datum
$\psi_0(y):=\psi(y,0)$ satisfies the following conditions:
\begin{equation}\label{moments}
M(\psi):=\int_{-\infty}^{\infty}\psi_0(y)\,dy\in(0,\infty), \quad
\varrho(\psi):=\int_{-\infty}^{\infty}|y^3\psi_0(y)|\,dy<\infty.
\end{equation}
Then, the following two results hold true:
\begin{equation}\label{asympt3}
\|\psi(y,t)-M(\psi)G(y,t)\|_{p}=O(t^{-2+1/2p}), \quad
\left\|\int_{-\infty}^y(\psi(s,t)-M(\psi)G(s,t))\,ds\right\|_{\infty}=O(t^{-3/2}).
\end{equation}
where $G(y,t)$ is the standard Gaussian
$$
G(y,t)=\frac{1}{\sqrt{4\pi t}}\exp\left(-\frac{y^2}{4t}\right).
$$
\end{lemma}
\noindent Lemma \ref{lem.Miller} is proved in \cite[Theorem 2]{MB}.
There, it is asked that $\psi_0\geq0$, but it is easy to check that
the result holds for any $\psi$ bounded. We have to check that
$\psi$ satisfies the condition \eqref{moments}. We have:
$$
M(\psi)=\int_{-\infty}^{\infty}\psi_0(y)\,dy=\lim\limits_{y\to\infty}(w_0(y)-w_0(-y))=K\in(0,\infty),
$$
and
\begin{equation*}
\begin{split}
\varrho(\psi)&=\int_{-\infty}^{\infty}|y^3v_{y}(y,0)|=\int_0^{\infty}\left|(\log r)^3ru_r(r,0)\right|\frac{dr}{r}\\
&=\int_0^{\infty}\left|(\log
r)^3u_{0,r}(r)\right|\,dr=\frac{1}{\omega_1}\int_{\real^N}\frac{|(\log|x|)^3|}{|x|^{N-1}}|\nabla
u_0(|x|)|\,dx<\infty,
\end{split}
\end{equation*}
due to the condition \eqref{initdata3}. Thus, by Lemma
\ref{lem.Miller}, we obtain the asymptotic convergence for $\psi$ as
in \eqref{asympt3}, with $M(\psi)=K$. We keep the second result in
\eqref{asympt3} and transform it to get
\begin{equation*}
\left\|w(y,t)-K\int_{-\infty}^{y}G(s,t)\,ds\right\|_{\infty}=O(t^{-3/2}),
\end{equation*}
or equivalently, recalling that $w(y,t)=K-v(y,t)$,
\begin{equation}\label{interm2}
\left\|v(y,t)-K\int_{y}^{\infty}G(s,t)\,ds\right\|_{\infty}=O(t^{-3/2}).
\end{equation}
Noticing that
$$
\int_{y}^{\infty}G(s,t)\,ds=\frac{1}{\sqrt{4\pi
t}}\int_{y}^{\infty}e^{-s^2/4t}\,ds=\frac{2\sqrt{t}}{2\sqrt{\pi
t}}\int_{y/(2\sqrt{t})}^{\infty}e^{-z^2}\,dz=\frac{1}{2}\hbox{erfc}\left(\frac{y}{2\sqrt{t}}\right),
$$
and finally undoing transformation \eqref{tr1} to get back to the
initial variables, we arrive to \eqref{asympt2} as desired.
\end{proof}
\noindent \textbf{Remarks.} (i) \textbf{Decay rates as
$t\to\infty$.} For $x\neq0$ fixed, the behavior as $t\to\infty$ of
$E(x,t)$, $F(x,t)$ is different with respect to the dimension $N$.
Indeed, when $N\geq3$, both profiles decay to 0 as $t\to\infty$ with
a rate
$$\frac{1}{\sqrt{t}}\exp\left(-\frac{(N-2)^2}{4}t\right),$$ as it is
easy to check from the explicit formulas. Since the two special
solutions make sense also for $N=1$ and $N=2$, we analyze the decay
rate also in these cases. For $N=2$, we have
$\lim\limits_{t\to\infty}E(x,t)=1$ and $F(x,t)=O(1/\sqrt{t})$ as
$t\to\infty$, and for $N=1$ we have
$\lim\limits_{t\to\infty}E(x,t)=2$ and $F(x,t)$ behaves in the same
way as for $N\geq3$.

\medskip

\noindent (ii) \textbf{Self-map.} Consider Eq. \eqref{eq1} in dimensions $N$, $\overline{N}$ respectively. Since
both are mapped to solutions of \eqref{1DHeat} through their correspondent transformation \eqref{tr1}, we have
$$
y=\log|x|+(N-2)t=\log|\overline{x}|+(\overline{N}-2)t,
$$
whence we obtain the mapping between radially symmetric solutions in
the two dimensions
\begin{equation}\label{selfmap}
\overline{u}(\overline{x},t)=u(x,t), \quad
|x|=|\overline{x}|e^{(\overline{N}-N)t}.
\end{equation}
The transformation \eqref{selfmap} is a \emph{self-map} of Eq.
\eqref{eq1} between dimensions $N$ and $\overline{N}$.

\medskip

\noindent (iii) \textbf{Hotspots. Inner and outer regimes.} Since $F(0,t)=0$ and $F(x,t)\to0$ as $t\to\infty$,
it is natural to think about the evolution of the maximum points of $F$ at time $t$ (called \emph{hotspots}). At
this point we use the second transformation \eqref{tr2}, in order to get, in the new variables $(z,t)$,
$$
F(x,t)=\overline{F}(z,t)=\frac{1}{\sqrt{4\pi t}}e^{-\frac{(N-2)^2}{4}t}e^{z-\frac{z^2}{(N-2)^2t}}.
$$
We omit the intermediate calculations which are straightforward.
Notice that the maximum lies at $z_0=(N-2)^2t/2$, and at that point,
$F(x,t)=1/\sqrt{4\pi t}$. Hence, the hotspots evolve with a decay
rate which is smaller than the "linear" decay rate as $t\to\infty$.
This is a typical feature of the presence of two different regimes
of convergence, one for $|x|$ small (the so-called \emph{inner
regime}) and another one for $|x|$ large (the so-called \emph{outer
regime}). This phenomenon exists also in literature for the porous
medium equation in domain with holes in dimension $N\geq3$
\cite{BQV} or for the $p$-Laplacian equation in dimension $N>p$
\cite{IV}. The difference between the behavior close to $x=0$ and
close to infinity will become more clear in the case $m>1$ that will
be studied in the forthcoming paper \cite{IS12}.

\section{Asymptotic convergence for general
solutions}\label{sec.asympt2}

We are now ready to prove Theorem \ref{th.2} in its maximal
generality.
\begin{proof}[Proof of Theorem \ref{th.2}] Let $u$ be a solution of
Eq. \eqref{eq1} with initial datum $u_0$ satisfying \eqref{initdata} and \eqref{initdata1}, with $u_0(0)=K>0$.
We construct the sub- and supersolutions $u_{-}(x,t)$, $u_{+}(x,t)$ that are the (radially symmetric) solutions
of Eq. \eqref{eq1} having initial data
\begin{equation}\label{subsuper} u_{-,0}(r):=\min\{u_0(x):|x|=r\}, \quad
u_{+,0}(r):=\max\{u_0(x):|x|=r\}, \end{equation}
which are obviously continuous and bounded, and both satisfy
\eqref{initdata1}, thus also \eqref{initdata1bis}. Thus, by the comparison principle, Corollary \ref{cor.comp},
we have
$$
u_{-}(x,t)\leq u(x,t)\leq u_{+}(x,t), \quad (x,t)\in\real^N\times(0,\infty).
$$
From part (a) of Theorem \ref{th.3} we know that \eqref{asympt0} holds true for $u_{-}$ and $u_{+}$. Since
$u_{-,0}(0)=u_{+,0}(0)=K$, it follows that
$$
\left|u(x,t)-\frac{K}{2}E(x,t)\right|\leq\max\left\{\left|u_{-}(x,t)-\frac{K}{2}E(x,t)\right|,
\left|u_{+}(x,t)-\frac{K}{2}E(x,t)\right|\right\}=O(t^{-1/2}),
$$
whence $u$ satisfies \eqref{asympt0}.
\end{proof}
At this point we also prove Proposition \ref{prop.3}.
\begin{proof}[Proof of Proposition \ref{prop.3}] (a) Let $u$ be a
radially symmetric solution with $u_0=M\delta_0$. Then, by \eqref{tr1}, its transformed $v$ will be a solution
to \eqref{1DHeat}, whose initial data will be the zero function, since $x=0$ is mapped into $y\to-\infty$. But
in our class of solutions, this solution of \eqref{1DHeat} is the zero function, by standard uniqueness results
\cite[Theorem 7, p. 58]{Evans}, which is a contradiction.

\noindent (b) Assume in a first step that $u$ is radially symmetric. We map $u(r,t)$ into $v(y,t)$ by
\eqref{tr1}. Then $v_0\equiv0$ in $(-\infty,\log r_0)$. Due to standard heat equation theory \cite[Section 2.3.3
(a)]{Evans}, $v(y,t)>0$ for any $t>0$, $y\in(-\infty,\log r_0)$, but it remains true that
$\lim\limits_{y\to-\infty}v(y,t)=0$, for any $t>0$. Coming back to $u$, we reach the conclusion.

Let now $u$ be a non-radial strong solution. We define the radially
symmetric functions $u_{-}$, $u_{+}$ as in \eqref{subsuper}. Since
$u_{-}(r,t)\leq u(x,t)$, it follows that $u(x,t)>0$ for any
$x\neq0$, $t>0$. Since $u(x,t)\leq u_{+}(r,t)$, it follows that
$u(0,t)=0$, for any $t>0$.
\end{proof}

\section{Extensions, comments and open problems}

\noindent \textbf{1. Dimension $N=1$.} In this case, the
transformation \eqref{tr1} applies to any solution. We thus notice
that the effect of the singularity at $x=0$ is just disconnecting
the real line. More precisely, let $u$ be a continuous solution to
\eqref{eq1} posed in dimension $N=1$ and such that $u(0,t)=0$ for
all $t>0$. Define
$$
u_{+}(x,t):=\left\{\begin{array}{ll}u(x,t), \quad {\rm if} \
x>0,\\0, \quad {\rm if} \ x\leq0,\end{array}\right. \quad
u_{-}(x,t):=\left\{\begin{array}{ll}u(x,t), \quad {\rm if} \
x<0,\\0, \quad {\rm if} \ x\geq0,\end{array}\right.
$$
and notice that, applying the transformation \eqref{tr1} to both
$u_{+}$ and $u_{-}$, we obtain two different solutions
$$
v_{+}(y,t):=u_{+}(x,t), \ y=\log\,x-t, \ x>0, \quad {\rm and} \quad
v_{-}(z,t)=u_{-}(x,t), \ z=\log(-x)-t, \ x<0,
$$
which will asymptotically converge towards two different Gaussian
profiles with different masses $M_{+}$ and $M_{-}$ as indicated in
the proof of Theorem \ref{th.1}. Coming back to initial variables,
the profile of $u$ will be
$$
U(x,t)=\left\{\begin{array}{ll} M_{+}F(x,t), \quad {\rm if} \ x>0,\\
M_{-}F(x,t), \quad {\rm if} x<0.\end{array}\right.
$$
In particular, we give a simple explicit example of solution having
two branches of the type indicated above:
\begin{equation}\label{interm7}
F_1(x,t):=\left\{\begin{array}{ll}\a F(x,t), \quad \hbox{for} \ x\le
0,\\(1-\a)F(x,t), \quad \hbox{for} \ x>0,\end{array}\right.
\end{equation}
for any $\a\in(0,1)$, $\a\neq1/2$. We leave the details of the
proofs of the previous statements to the reader, along the lines of
the proof of Theorem \ref{th.1}. For initial data with $u_0(0)=K>0$,
Theorem \ref{th.2} still applies in this case. In order to show the
difference of behavior with respect to $N\geq3$, we plot in Figure
\ref{figure2} the two tipical profiles for dimension $N=1$.

\begin{figure}[ht!]
  \begin{center}
  \includegraphics[width=15cm,height=10cm]{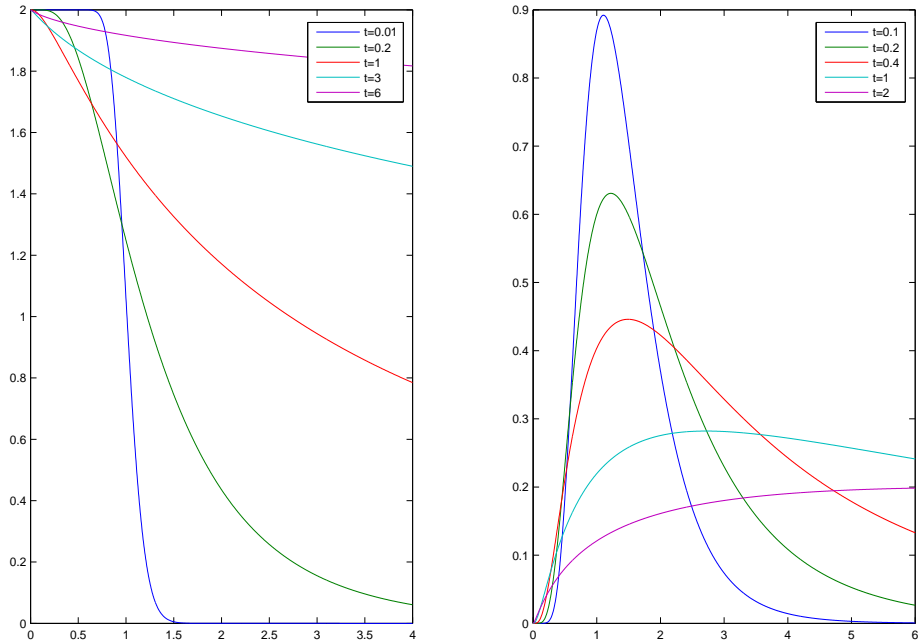}
  \end{center}
  \caption{Profiles $E$ and $F$ for dimension $N=1$ at various times.} \label{figure2}
\end{figure}

\medskip

\noindent \textbf{2. Improvements and open problems related to
Theorems \ref{th.2} and \ref{th.3}}. Besides the conditions in
Theorem \ref{th.3}, in the case $u_0(0)=K$, one can improve the rate
of convergence to the asymptotic profile $E(x,t)$ by using, for
example, results from the paper \cite{Y07}. We leave these
improvements to the reader, along the lines of the proof of Theorem
\ref{th.3}. On the other hand, an open problem arises naturally from
the results in Theorems \ref{th.2} and \ref{th.3}. That is, can one
prove \eqref{asympt0} for \textbf{general} (that means, not
necessarily radially symmetric) solutions, only asking the condition
\eqref{initdata1bis} to hold?

\medskip

\noindent \textbf{3. Asymptotic profiles of non-radially symmetric
solutions with $u_0(0)=0$.} Theorem \ref{th.1} and the
counterexamples following it raise a very natural, but also in our
opinion very difficult problem to classify the asymptotic profiles
for general solutions of Eq. \eqref{eq1} when $u_0(0)=0$. Since the
counterexamples are constructed in a very specific way, a first
partial question that could be raised is: are there all the profiles
weighted combinations of the radial one? For dimension $N=1$ we give
an answer in the first comment of this section, but the problem
remains open for the rest of dimensions.

\medskip

\noindent \textbf{4. General densities}. As we have said in the
Introduction, there exists an increasing interest for the more
general problem \eqref{eq.general}. This case has not been yet
studied in detail with $\varrho(x)\sim|x|^{-2}$ as $|x|\to\infty$
and $m=1$, but it has been strongly studied for $m>1$, \cite{RV08,
RV09, KRV10} and references therein. The authors consider there
densities $\varrho$ which are regular near $x=0$ and show that the
asymptotic behavior is given by the fundamental solution to
\eqref{eq2}. It might be interesting to raise also the problem of
considering general densities of various other forms, for example
$$
\varrho(x)\sim|x|^{-2} \ \hbox{as} \ |x|\to0, \quad
\lim\limits_{|x|\to\infty}\varrho(x)=C\in(0,\infty),
$$
that is, preserving the other property of the pure power density
(the singularity at $x=0$) and renouncing to the first one (the
behavior as $|x|\to\infty$).

\bigskip

\textsc{Acknowledgements.} R. I. partially supported by the Spanish
projects MTM2008-03176 and MTM2012-31103. A. S. partially supported
by the Spanish projects MTM2008-06326-C02-02 and MTM2011-25287. Part
of this work has been done during visits by R. I. to the Univ.
Carlos III and Univ. Rey Juan Carlos in Madrid.

\bibliographystyle{plain}

\end{document}